\documentclass[11pt,reqno]{my-glm-t}

\usepackage{amsmath,amsthm,amscd,amsfonts,amssymb,graphicx,color,xcolor}
\RequirePackage[backref,colorlinks=true]{hyperref}

\begin{document}

\title{{\color{green!42!blue}\LARGE{\bf 
Local existence and uniqueness for a fractional SIRS model with Mittag--Leffler law}}}

\author[a1]{Moulay Rchid Sidi Ammi\corref{c1}}
\ead{rachidsidiammi@yahoo.fr}

\author[a2]{Mostafa Tahiri}
\ead{my.mustafa.tahiri@gmail.com}

\author[a3]{Delfim F. M. Torres}
\ead{delfim@ua.pt}

\address[a1]{Department of Mathematics, AMNEA Group, Faculty of Sciences and Techniques,\\
Moulay Ismail University of Meknes, B.~P. 509, Errachidia, Morocco. \vskip 0.1cm}

\address[a2]{Department of Mathematics, AMNEA Group, Faculty of Sciences and Techniques,\\
Moulay Ismail University of Meknes, B.~P. 509, Errachidia, Morocco. \vskip 0.1cm}

\address[a3]{Center for Research and Development in Mathematics and Applications (CIDMA),\\
Department of Mathematics, University of Aveiro, 3810-193 Aveiro, Portugal.}

\cortext[c1]{Corresponding author}


\setcounter{page}{61}
\vol{10(2) (2021)} 
\pages{61--71}

\recivedat{30 Apr 2021; Revised: 25 May 2021; Accepted: 25 Jun 2021}

\authors{M. R. Sidi Ammi, M. Tahiri, D. F. M. Torres}

\doi{\href{https://doi.org/10.31559/glm2021.10.2.7}{https://doi.org/10.31559/glm2021.10.2.7}}


\begin{abstract}
In this paper, we study an epidemic model with Atangana--Baleanu--Caputo 
(ABC) fractional derivatives. We obtain a special solution using an iterative 
scheme via Laplace transformation. Uniqueness and existence of solution 
using the Banach fixed point theorem are studied. A detailed analysis 
of the stability of the special solution is presented. Finally, 
our generalized model in the ABC derivative sense is solved 
numerically by the Adams--Bashforth--Moulton method.

\end{abstract}

\maketitle


\newtheorem{theorem}{Theorem}[section]
\newtheorem{lemma}[theorem]{Lemma}
\newtheorem{proposition}[theorem]{Proposition}
\newtheorem{corollary}[theorem]{Corollary}
\newtheorem{question}[theorem]{Question}

\theoremstyle{definition}
\newtheorem{definition}[theorem]{Definition}
\newtheorem{algorithm}[theorem]{Algorithm}
\newtheorem{conclusion}[theorem]{Conclusion}
\newtheorem{problem}[theorem]{Problem}

\theoremstyle{remark}
\newtheorem{remark}[theorem]{Remark}
\numberwithin{equation}{section}


\section{Introduction}
\label{sec1}

The generalization of mathematical models in epidemiology 
has the purpose of providing a good description, closer 
to reality, by using a proper concept of derivatives, 
particularly a notion of differentiation of fractional order. 
Since the fractional order can be any positive real $\alpha$, 
one can choose the one that better fits available data \cite{MR3719831}.
Therefore, we can adjust the model to real data for better predict 
the future evolution of the disease \cite{MyID:419,MR3789859}. 
Moreover, virus propagation is typically discontinuous and 
classical differential models cannot describe them in a proper way.
In contrast, fractional systems deal naturally with such discontinuous 
properties \cite{MR3673702,MR3854267}.

The virus propagation is similar to heat transmission or moistness penetrability 
in a porous medium, which can be exactly modelled by fractional calculus \cite{ref35,ref34}.
Authors in \cite{ref36,ref37} gave a geometrical description of fractional calculus, 
concluding that the fractional order can be related with the fractal dimension. 
The relationship between fractal dimension and fractional calculus has been reported 
by several different authors: see \cite{MR3320677,ref38} and references therein. 
The fractional complex transform \cite{ref39,ref40} is an approximate transform 
of a fractal space (time) to a continuous one, and it is now widely used in 
fractional calculus \cite{MR3571716,ref41,ref42}.

Several definitions were proposed in the literature for fractional differentiation. 
Some authors have used the Caputo fractional derivative because of some useful 
properties provided by this derivative \cite{ref17,ref19,ref22,ref25}, 
especially in the analysis of the spread of diseases \cite{ref10,ref21,ref24,ref26}. 
However, this derivative has some limitations, for instance, the kernel therein 
has a singularity. To solve this problem, Caputo and Fabrizio have proposed 
a derivative with fractional order that has a kernel with no singularity. Recently, 
the Caputo--Fabrizio derivative was used by few researchers to solve 
some real world problems, see for example \cite{ref18,ref20}. However, 
many other researchers testified that the Caputo--Fabrizio operator 
is nothing more than a filter with a fractional regulator. They based their 
argument upon the fact that the kernel used in this design is local and 
the associate integral is the average of the given function and its integral. 
In many applications of fractional differentiation based on the power law $x^{-\sigma}$, 
the Mittag--Leffler function is mostly present. The Mittag--Leffler function is 
a generalization of the exponential function. In addition, it is also a non-local 
kernel \cite{ref3,ref16,ref28}. To solve the failures of the Caputo--Fabrizio 
derivative, the fractional derivative based on the Mittag--Leffler function 
was introduced and used in some new problems with great success 
\cite{ref29,ref30,ref31,ref32}. It is important to know that where 
the power based on the $x^{-\sigma}$ function relaxes then the Mittag--Leffler 
function raises more complex and interesting dynamics.

Several research papers have been published using this new concept 
of fractional differentiation with Mittag--Leffler function. The results 
obtained in \cite{ref3} reveal that the new concept is more adequate 
for modelling real world problems, to take into account non-local phenomena 
as well as memory effects.

Some infectious diseases confer temporally acquired immunity. 
This type of diseases can be modeled by the SIRS model. The total population 
$N$ is divided into three compartments with $N = S + I + R$, where $S$ 
is the number of susceptible, $I$ is the number of infectious individuals 
and $R$ is the number of recovered individuals \cite{ref9}.

The main aim of our work is to show, by applying the Picard iteration method, 
the existence and uniqueness of solution to the Atangana--Baleanu--Caputo 
fractional derivative SIRS epidemic model with a saturated treatment 
and an incidence function as presented in \cite{ref9}, which generalizes
the classical bilinear incidence rate, the saturated incidence rate, 
the Beddington--DeAngelis functional response \cite{ref11,ref23} 
and the Crowley--Martin functional response \cite{ref12}.

Picard iteration has more theoretical value than practical one. 
Indeed, finding an approximate solution using this method 
is almost impractical for complicated functions of the second member 
of a fractional differential equation. For that, in this work we will 
use a numerical Adams--Bashforth--Moulton method \cite{ref27}.

The paper is organized as follows. In Section~\ref{sec2}, we give the 
definitions of the new fractional derivative with non-singular and non-local kernel. 
The existence of solution for our epidemic model via Picard-Lindel$\ddot{o}$f method 
is investigated in Section~\ref{sec3}. In Section~\ref{sec4}, we provide a
stability analysis of the numerical scheme obtained by Picard's iteration method. 
Finally, in Section~\ref{sec5}, some numerical results obtained at different 
instances of fractional order are presented.


\section{Basic properties of the new fractional derivative}
\label{sec2}

In this section, we present the definitions of the new fractional derivative 
with non-singular and non-local kernel \cite{ref3}.

\begin{definition} 
\label{def2.1}
Let $f \in H^1(0,T)$, $T>0$, $0 \leq \alpha < 1$. The 
Atangana--Baleanu fractional derivative in Caputo sense is given by
\begin{equation} 
\label{eq2.1}
^{ABC}_{0}D{_t^{\alpha}} (f(t)) = \dfrac{B(\alpha)}{1-\alpha} 
\int_{0}^t f' (s) E_{\alpha}\bigg[-\alpha \dfrac{(t-s)^{\alpha}}{1-\alpha}\bigg]ds,
\end{equation}
where the kernel $E_{\alpha}$ is the Mittag--Leffler function of one parameter 
and $B(\alpha)$ is a normalization function such that  $B(0)=B(1)=1$ \cite{ref4}.
\end{definition}

\begin{definition}
\label{def2.2} 
The fractional integral of order $\alpha$ is defined by 
\begin{equation} 
\label{eq2.2}
^{AB}_{0}I_{t}^\alpha (f(t)) = \dfrac{1-\alpha}{B(\alpha)}f(t) 
+ \dfrac{\alpha}{B(\alpha)\Gamma(\alpha)}\int_{0}^t f(u) (t-u)^{\alpha -1}du.
\end{equation}
Here $\Gamma(\cdot)$ is the Euler Gamma function, 
which is defined as $\Gamma(x)=\displaystyle \int_0^\infty t^{x-1} e^{-t} dt$.
\end{definition}

\begin{remark} 
\label{rk2.1}
When $\alpha=0$ in \eqref{eq2.2}, the initial function is obtained; 
and when $\alpha=1$ we obtain the ordinary integral.
\end{remark}


\section{Existence and uniqueness of solution}
\label{sec3}

In this section, we extend the SIRS model to a fractional-order model. 
The reasons and motivations for such extension have been presented 
in the introduction. Nevertheless, it is important noting that the concept 
of local derivative that is used to describe the rate of change has failed 
to model accurately some complex real-world problems. Due to this failure, 
the concept of fractional differentiation, based on the convolution 
of $x^{-\sigma}$, was introduced, and also failed in some cases due to the 
disc of convergence of this function. The Mittag--Leffler function 
can be used in order to handle more physical problems.

We consider the following system:
\begin{equation}
\label{eq3.1}
\begin{aligned}
^{ABC}_{0}D{_t^{\alpha}}S(t) &= \Lambda-\mu S(t)
-\dfrac{\beta S(t) I(t)}{1+k_1 S(t)+k_2 I(t)+k_3 S(t) I(t)}+\lambda R(t),\\[2ex]
^{ABC}_{0}D{_t^{\alpha}}I(t) &= \dfrac{\beta S(t) I(t)}{1+k_1 S(t)+k_2 I(t)
+k_3 S(t) I(t)}-(\mu+d+r)I(t)-\dfrac{\xi I(t)}{1+\gamma I(t)},\\[2ex]
^{ABC}_{0}D{_t^{\alpha}}R(t) &= r I(t)+\dfrac{\xi I(t)}{1+\gamma I(t)}-(\mu+\lambda)R(t),
\end{aligned}
\end{equation}
subject to initial conditions
\begin{equation}
\label{eq3.2}
S(0)\geq 0,\quad I(0)\geq0, \quad R(0)\geq 0.
\end{equation}

The positive constants $\Lambda$, $\beta$, $\mu$, $r$, $d$, $\lambda$  
are the recruitment rate of the population, the infection rate, 
the natural death rate, the recovery rate of the infective individuals, 
the death rate due to disease, the rate that recovered individuals 
lose immunity and return to the susceptible class, respectively.  
While contacting with infected individuals, the susceptible become 
infected at the incidence rate $\beta SI/(1+k_1 S+k_2 I+k_3 SI)$, 
with $k_1$, $k_2$ and $k_3$ non-negative constants \cite{ref9}.
Through treatment, the infected individuals recover at a saturated treatment 
function $\xi I(t)/(1+\gamma I(t))$, where $\xi$ is positive, $\gamma$ 
is nonnegative, and $1/(1+\gamma I(t))$ describes the reverse effect 
of the infected being delayed for treatment. When $\gamma=0$, 
the saturated treatment function stays to the linear one \cite{ref6}.


\subsection{Iterative scheme with Laplace transform}
\label{3.1}

\begin{theorem} 
\label{th3.1.1}
For $\alpha \in [0,1]$, the following 
time fractional ordinary differential equation
\begin{equation} \label{eq3.1.1}
^{ABC}_{0}D{_t^{\alpha}} (f(t)) = u(t),
\end{equation}
has a unique solution, namely
\begin{equation} \label{eq3.1.2}
f(t) = f(0)+\dfrac{1-\alpha}{B(\alpha)}u(t) 
+ \dfrac{\alpha}{B(\alpha)\Gamma(\alpha)}\int_{0}^t u(p) (t-p)^{\alpha -1}dp.
\end{equation} 	
\end{theorem}

\begin{proof}
Using the Laplace transformation on both sides of the equation \eqref{eq3.1.1}
we get
\begin{equation*}
L\big[^{ABC}_{0}D{_t^{\alpha}} (f(t))\big](p) = L\big[u(t)\big](p), \hspace{1cm} p>0
\end{equation*}
and, according to \cite[Theorem 3]{ref16}, we have that
\begin{gather*}
\dfrac{B(\alpha)}{1-\alpha} \dfrac{p^{\alpha}L\{f(t)\}(p)
-p^{\alpha-1}f(0)}{p^{\alpha}+\dfrac{\alpha}{1-\alpha}} = L\{u(t)\}(p),
\end{gather*}
which is equivalent to
\begin{gather*}
L\{f(t)\}(p) = \dfrac{1}{p}f(0) + \dfrac{1-\alpha}{B(\alpha)}L\{u(t)\}(p) 
+ \dfrac{\alpha}{p^{\alpha}B(\alpha)}L\{u(t)\}(p).
\end{gather*}
Now, we use the inverse Laplace transform
\begin{gather*}
f(t) = f(0) + \dfrac{1-\alpha}{B(\alpha)}u(t) 
+ L^{-1}\bigg\{\dfrac{\alpha}{p^{\alpha}B(\alpha)}L\{u(t)\}(p)\bigg\}(t)
\end{gather*}
to obtain
\begin{gather*}
\dfrac{\alpha}{p^{\alpha}B(\alpha)}=\dfrac{\alpha}{B(\alpha)}
L\bigg\{\dfrac{t^{\alpha-1}}{\Gamma(\alpha)}\bigg\}(p).
\end{gather*}
Let $F(p)=\dfrac{\alpha}{B(\alpha)}L\bigg\{\dfrac{t^{\alpha-1}}{\Gamma(\alpha)}\bigg\}(p)$ 
and $G(p)=L\{u(t)\}(p)$. Then, applying the convolution theorem, we obtain
\begin{gather*}
\begin{aligned}
L^{-1}\bigg\{\dfrac{\alpha}{p^{\alpha}B(\alpha)}L\{u(t)\}(p)\bigg\}(t)
&=L^{-1}\bigg\{F(p)\times G(p) \bigg\}(t)=\dfrac{\alpha}{B(\alpha)}
\bigg(u(t)*\dfrac{t^{\alpha-1}}{\Gamma(\alpha)}\bigg)(p)\\
\quad&=\dfrac{\alpha}{B(\alpha)\Gamma(\alpha)}\int_{0}^t u(p) (t-p)^{\alpha -1}dp.
\end{aligned}
\end{gather*}
Hence, the result is proved.
\end{proof}

Using Theorem~\ref{th3.1.1}, our system is equivalent to 
\begin{gather}
\label{eq3.1.3}
\begin{aligned}
S(t)-S(0)&=\dfrac{1-\alpha}{B(\alpha)}\bigg\{\Lambda-\mu S(t)
-\dfrac{\beta S(t) I(t)}{1+k_1 S(t)+k_2 I(t)+k_3 S(t) I(t)}+\lambda R(t)\bigg\}\\
&\quad +\dfrac{\alpha}{B(\alpha)\Gamma(\alpha)}\int_{0}^t (t-p)^{\alpha -1}\bigg\{\Lambda-\mu S(p)\\
&\quad -\dfrac{\beta S(p) I(p)}{1+k_1 S(p)+k_2 I(p)+k_3 S(p) I(p)}+\lambda R(p)\bigg\} dp,\\
I(t)-I(0)&=\dfrac{1-\alpha}{B(\alpha)}\bigg\{\dfrac{\beta S(t) I(t)}{1+k_1 S(t)+k_2 I(t)+k_3 S(t) I(t)}-(\mu+d+r)I(t)\\
&\quad -\dfrac{\xi I(t)}{1+\gamma I(t)}\bigg\}-\dfrac{\alpha}{B(\alpha)\Gamma(\alpha)}
\int_{0}^t (t-p)^{\alpha -1}\bigg\{(\mu+d+r)I(p)\\
&\quad +\dfrac{\xi I(p)}{1+\gamma I(p)}-\dfrac{\beta S(p) I(p)}{1+k_1 S(p)+k_2 I(p)+k_3 S(p) I(p)}\bigg\}dp,\\
R(t)-R(0)&=\dfrac{1-\alpha}{B(\alpha)}\bigg\{r I(t)+\dfrac{\xi I(t)}{1+\gamma I(t)}-(\mu+\lambda)R(t)\bigg\}\\
&\quad +\dfrac{\alpha}{B(\alpha)\Gamma(\alpha)}\int_{0}^t (t-p)^{\alpha -1}\bigg\{r I(p)
+\dfrac{\xi I(p)}{1+\gamma I(p)}-(\mu+\lambda)R(p)\bigg\}dp.
\end{aligned}
\end{gather}

The iterative scheme of the system \eqref{eq3.1.3} is given by
\begin{equation}
\label{eq3.1.4}
S_0(t)=S(0);\hspace{0.5cm}
I_0(t)=I(0);\hspace{0.5cm}
R_0(t)=R(0).
\end{equation}
\begin{equation}
\label{eq3.1.5}
\begin{aligned}
S_{n+1}(t)&=\dfrac{1-\alpha}{B(\alpha)}\bigg\{\Lambda-\mu S_n(t)
-\dfrac{\beta S_n(t) I_n(t)}{1+k_1 S_n(t)+k_2 I_n(t)+k_3 S_n(t) I_n(t)}+\lambda R_n(t)\bigg\}\\
&\quad +\dfrac{\alpha}{B(\alpha)\Gamma(\alpha)}\int_{0}^t (t-p)^{\alpha -1}\bigg\{\Lambda-\mu S_n(p)\\
&\quad -\dfrac{\beta S_n(p) I_n(p)}{1+k_1 S_n(p)+k_2 I_n(p)+k_3 S_n(p) I_n(p)}+\lambda R_n(p)\bigg\} dp,\\
I_{n+1}(t)&=\dfrac{1-\alpha}{B(\alpha)}\bigg\{\dfrac{\beta S_n(t) I_n(t)}{1
+k_1 S_n(t)+k_2 I_n(t)+k_3 S_n(t) I_n(t)}-(\mu+d+r)I_n(t)\\
&\quad -\dfrac{\xi I_n(t)}{1+\gamma I_n(t)}\bigg\}-\dfrac{\alpha}{B(\alpha)\Gamma(\alpha)}
\int_{0}^t (t-p)^{\alpha -1}\bigg\{(\mu+d+r)I_n(p)\\
&\quad +\dfrac{\xi I_n(p)}{1+\gamma I_n(p)}-\dfrac{\beta S_n(p) I_n(p)}{1+k_1 S_n(p)
+k_2 I_n(p)+k_3 S_n(p) I_n(p)}\bigg\}dp,\\
R_{n+1}(t)&=\dfrac{1-\alpha}{B(\alpha)}\bigg\{r I_n(t)
+\dfrac{\xi I_n(t)}{1+\gamma I_n(t)}-(\mu+\lambda)R_n(t)\bigg\}\\
&\quad +\dfrac{\alpha}{B(\alpha)\Gamma(\alpha)}\int_{0}^t 
(t-p)^{\alpha -1}\bigg\{r I_n(p)+\dfrac{\xi I_n(p)}{1+\gamma I_n(p)}
-(\mu+\lambda)R_n(p)\bigg\}dp.
\end{aligned}
\end{equation}

If we take the limit, we expect to obtain the exact solution.


\subsection{Existence of solution via Picard--Lindel$\ddot{o}$f method}
\label{sec3.2}

We define the operators
\begin{gather}
\begin{aligned}
f_1(t,\Omega(t)) &= \Lambda-\mu S(t)
-\dfrac{\beta S(t) I(t)}{1+k_1 S(t)+k_2 I(t)+k_3 S(t) I(t)}+\lambda R(t),\\
f_2(t,\Omega(t)) &= \dfrac{\beta S(t) I(t)}{1+k_1 S(t)+k_2 I(t)
+k_3 S(t) I(t)}-(\mu+d+r)I(t)-\dfrac{\xi I(t)}{1+\gamma I(t)},\\
f_3(t,\Omega(t)) &= r I(t)+\dfrac{\xi I(t)}{1+\gamma I(t)}-(\mu+\lambda)R(t),
\end{aligned}
\end{gather}
and the matrix form of system \eqref{eq3.1} subject to conditions \eqref{eq3.2}:
\begin{equation}
\label{eq3}
\begin{aligned}
^{ABC}_{0}D{_t^{\alpha}}\Omega(t)=F(t,\Omega(t)),\hspace{0.5cm}
\Omega(0)=\Omega_0,
\end{aligned}
\end{equation}
where $\Omega(t)=\big(S(t),I(t),R(t)\big)$, $\Omega_0=(S_0,I_0,R_0)$
and $F(t,\Omega(t))=\big(f_1(t,\Omega(t)),f_2(t,\Omega(t)),f_3(t,\Omega(t))\big)$.

\begin{lemma}
\label{lem3.2.1}
The function $F$ is Lipschitz continuous on $[0,T] \times \mathbf{B}(\Omega_0,\rho)$ with 
$$
[0,T] \times \mathbf{B}(\Omega_0,\rho)=\{(t,\Omega(t))
\in [0,T] \times\mathbb{R}_+^{3} /  
\underset{t\in [0, T]}{sup}\|\Omega(t)-\Omega_0\|_1 \leq \rho\},
$$
i.e., there exists a constant $L\in\mathbb{R}_+$, $\forall (t,\Omega_1(t)), 
(t,\Omega_2(t))\in [0,T] \times \mathbf{B}(\Omega_0,\rho)$
\begin{gather}
\|F(t, \Omega_1(t))-F(t, \Omega_2(t))\|_1\leq L\|\Omega_1(t)-\Omega_2(t)\|_1,
\end{gather}
\end{lemma}
where $\|\Omega(t)\|_1 = \sum_{i=1}^{3}|\Omega_i(t)|$ is the Manhattan norm.

\begin{proof}
We shall prove that $F$ satisfies the Lipschitz condition 
in the second argument $\Omega$:
\begin{gather}
\begin{aligned}
&\|F(t, \Omega_1(t))-F(t, \Omega_2(t))\|_1 
=  |f_1(t, \Omega_1(t))-f_1(t, \Omega_2(t))|+|f_2(t, \Omega_1(t))\\
\quad&-f_2(t, \Omega_2(t))|+|f_3(t, \Omega_1)-f_3(t, \Omega_2)|\\
\quad&=\bigg|\Lambda-\mu S_1(t)-\dfrac{\beta S_1(t) I_1(t)}{1
+k_1 S_1(t)+k_2 I_1(t)+k_3 S_1(t) I_1(t)}+\lambda R_1(t)\\
\quad& -\big(\Lambda-\mu S_2(t)-\dfrac{\beta S_2(t) I_2(t)}{1
+k_1 S_2(t)+k_2 I_2(t)+k_3 S_2(t) I_2(t)}+\lambda R_2(t)\big)\bigg|\\
\quad&+\bigg|\dfrac{\beta S_1(t) I_1(t)}{1+k_1 S_1(t)+k_2 I_1(t)
+k_3 S_1(t) I_1(t)}-(\mu+d+r)I_1(t)-\dfrac{\xi I_1(t)}{1+\gamma I_1(t)}\\
\quad&-\bigg(\dfrac{\beta S_2(t) I_2(t)}{1+k_1 S_2(t)+k_2 I_2(t)
+k_3 S_2(t) I_2(t)}-(\mu+d+r)I_2(t)-\dfrac{\xi I_2(t)}{1+\gamma I_2(t)}\bigg)\bigg|\\
\quad&+\bigg|r I_1(t)+\dfrac{\xi I_1(t)}{1+\gamma I_1(t)}-(\mu+\lambda)R_1(t)
-\big(r I_2(t)+\dfrac{\xi I_2(t)}{1+\gamma I_2(t)}-(\mu+\lambda)R_2(t)\big)\bigg|.
\end{aligned}
\end{gather}
We reduce the next two fractions to the same denominator
$$
D=\big(1+k_1 S_1(t)+k_2 I_1(t)+k_3 S_1(t) I_1(t)\big)\big(1
+k_1 S_2(t)+k_2 I_2(t)+k_3 S_2(t) I_2(t)\big).
$$
We note that $D>1$, hence
\begin{gather}
\begin{aligned}
&\bigg|\dfrac{\beta S_1(t) I_1(t)}{1+k_1 S_1(t)+k_2 I_1(t)
+k_3 S_1(t) I_1(t)}-\dfrac{\beta S_2(t) I_2(t)}{1+k_1 S_2(t)
+k_2 I_2(t)+k_3 S_2(t) I_2(t)}\bigg|\\
&\quad \leq\beta|S_1(t)I_1(t)-S_2(t)I_2(t)|+\beta k_1|S_1(t)||S_2(t)||I_1(t)-I_2(t)|\\ 
&\quad+\beta k_2|I_1(t)||I_2(t)||S_1(t)-S_2(t)|\\
&\quad \leq \beta|S_1(t)I_1(t)-S_2(t)I_1(t)|+\beta|S_2(t)I_1(t)-S_2(t)I_2(t)|\\
&\quad+\beta k_1|S_1(t)||S_2(t)||I_1(t)-I_2(t)| +\beta k_2|I_1(t)||I_2(t)||S_1(t)-S_2(t)|\\
&\quad \leq \beta|I_2(t)||S_1(t)-S_2(t)|+\beta|S_1(t)||I_1(t)-I_2(t)|\\
&\quad+\beta k_1|S_1(t)||S_2(t)||I_1(t)-I_2(t)| +\beta k_2|I_1(t)||I_2(t)||S_1(t)-S_2(t)|.
\end{aligned}
\end{gather}
Similarly, we can prove that
\begin{gather}
\begin{aligned}
&\bigg|\dfrac{\xi I_1(t)}{1+\gamma I_1(t)}-\dfrac{\xi I_2(t)}{1+\gamma I_2(t)}\bigg|
\leq \xi|I_1(t)-I_2(t)|.
\end{aligned}
\end{gather}
Then,
\begin{gather}
\begin{aligned}
&\|F(t, \Omega_1(t))-F(t, \Omega_2(t))\|_1\\
\quad&\leq \mu |S_1(t)-S_2(t)| +\lambda|R_1(t)-R_2(t)|+\beta|I_2(t)||S_1(t)-S_2(t)|\\
\quad&+\beta|S_1(t)||I_1(t)-I_2(t)|\\
\quad&+\beta k_1|S_1(t)||S_2(t)||I_1(t)-I_2(t)| +\beta k_2|I_1(t)||I_2(t)||S_1(t)-S_2(t)|\\
\quad&+\beta|I_2(t)||S_1(t)-S_2(t)|+\beta|S_1(t)||I_1(t)-I_2(t)|\\
\quad&+\beta k_1|S_1(t)||S_2(t)||I_1(t)-I_2(t)| +\beta k_2|I_1(t)||I_2(t)||S_1(t)-S_2(t)|\\
\quad&+(\mu+d+r)|I_1(t)-I_2(t)|+\xi|I_1(t)-I_2(t)|\\
\quad&+r|I_1(t)-I_2(t)|+(\mu+\lambda)|R_1(t)-R_2(t)|+\xi|I_1(t)-I_2(t)|\\
\quad&\leq(\mu+2\beta|I_2(t)|+2\beta k_2|I_1(t)||I_2(t)|)|S_1(t)-S_2(t)|\\
\quad&+(2\beta|S_1(t)|+2\beta k_1|S_1(t)||S_2(t)|+\mu+d+2r+2\xi)|I_1(t)-I_2(t)|\\
\quad&+(2\lambda+\mu)|R_1(t)-R_2(t)|
\end{aligned}
\end{gather}
\begin{gather*}
\begin{aligned}
&\leq(\mu+2\beta (\rho+I_0)+2\beta k_2 (\rho+I_0)^2)|S_1(t)-S_2(t)|+(2\beta (\rho+S_0)\\
\quad&+2\beta k_1 (\rho+S_0)^2+\mu +d+2r+2\xi)|I_1(t)-I_2(t)|+(2\lambda+\mu)|R_1(t)-R_2(t)|\\
\quad&\leq L\|\Omega_1(t)-\Omega_2(t)\|_1,
\end{aligned}
\end{gather*}
where
\begin{gather}
\begin{aligned}
L=Max\big\{\big(&\mu+2\beta (\rho+I_0)+2\beta k_2 (\rho+I_0)^2\big), 
\big(2\beta (\rho+S_0)+2\beta k_1 (\rho+S_0)^2+\mu\\
\quad&+d+2r+2\xi\big), \big(2\lambda+\mu\big)\big\}.
\end{aligned}
\end{gather}
It is clear that $L>0$. Then, $F$ is Lipschitz continuous 
in the second argument.
\end{proof}

\begin{theorem} 
\label{th3.2.1}
Let $W=\big\{\Omega\in (C^0[0, \delta])^3 : 
\Omega(0)=\Omega_0 , \Omega(t)\in \mathbf{B}(\Omega_0,\rho)\big\}$
be the space of continuous functions which is complete with norm 
$\| \Omega\|_W = \underset{t\in [0, \delta]}{sup}\|\Omega(t)\|_1$. 
Let $\| F(t,\Omega)\|_W\leq \bar N$ for $(t, \Omega)\in [0, \delta]\times\mathbf{B}(\Omega_0,\rho)$. 
Then the system \eqref{eq3.1} subject to conditions \eqref{eq3.2} has a unique solution 
on $[0, \delta]$ with  $\delta>0$ and 
$$
\delta < min\bigg\{T, \bigg(\dfrac{\rho B(\alpha)\Gamma(\alpha)}{\bar N}
+\alpha\Gamma(\alpha)-\Gamma(\alpha)\bigg)^{\alpha^{-1}}, 
\bigg(\dfrac{B(\alpha)\Gamma(\alpha)}{L}+\alpha\Gamma(\alpha)
-\Gamma(\alpha)\bigg)^{\alpha^{-1}}\bigg\}.
$$
\end{theorem}

\begin{proof}	
The fixed-point theorem in Banach space $W$ can be employed here.
For that, the Picard's operator $\Theta$ is defined between 
the functional space $W$ into itself, as follows:
\begin{gather}
\Theta : W \rightarrow W,
\end{gather}
such that
\begin{equation}
\label{eq3.2.1}
\Theta\Omega(t)=\Omega_0 + \dfrac{1-\alpha}{B(\alpha)}F(t,\Omega(t)) 
+ \dfrac{\alpha}{B(\alpha)\Gamma(\alpha)}\int_{0}^t F(p,\Omega(p)) (t-p)^{\alpha -1}dp.
\end{equation}
We show that $\Theta$ maps $W$ to $W$. Due to the fact that there is no disease 
that is able to kill the whole world population, we can suppose that $\Omega(t)\in W$. 
We can also assume that the solution is bounded within a period of time since the number 
of targeted population is finite. Next, we show that $\Theta\Omega(t)\in W$:
\begin{gather}
\begin{aligned}
\| \Theta\Omega-\Omega_0\|_W
&= \underset{t\in[0, \delta]}{sup} \bigg\|\dfrac{1-\alpha}{B(\alpha)}F(t,\Omega(t)) 
+ \dfrac{\alpha}{B(\alpha)\Gamma(\alpha)}\int_{0}^t F(p,\Omega(p)) (t-p)^{\alpha -1}dp\bigg\|_1\\
\quad&\leq  \dfrac{1-\alpha}{B(\alpha)}\times\underset{t\in[0, \delta]}{sup}\big\|F(t,\Omega(t))\big\|_1 
+ \dfrac{\alpha}{B(\alpha)\Gamma(\alpha)}\\
&\quad\times\underset{t\in[0, \delta]}{sup}\int_{0}^t \big\|F(p,\Omega(p))\big\|_1 (t-p)^{\alpha -1}dp\\
\quad&\leq  \dfrac{1-\alpha}{B(\alpha)}\times\underset{t\in[0, \delta]}{sup}\big\|F(t,\Omega(t))\big\|_1 
+ \dfrac{\alpha}{B(\alpha)\Gamma(\alpha)}\\
&\quad\times\int_{0}^\delta \underset{p\in[0, \delta]}{sup}\big\|F(p,\Omega(p))\big\|_1 (t-p)^{\alpha -1}dp\\
\quad&\leq \dfrac{(1-\alpha)}{B(\alpha)}\bar N 
+ \dfrac{\alpha}{B(\alpha)\Gamma(\alpha)}\bar N\int_{0}^\delta (t-p)^{\alpha -1}dp\\
\quad&\leq \dfrac{(1-\alpha)}{B(\alpha)}\bar N 
+ \dfrac{\alpha}{B(\alpha)\Gamma(\alpha)}\bar N\int_{0}^\delta (t-p)^{\alpha -1}dp\\
\quad&\leq \dfrac{(1-\alpha)}{B(\alpha)}\bar N 
+ \dfrac{\delta^\alpha}{B(\alpha)\Gamma(\alpha)}\bar N.
\end{aligned}
\end{gather}
Then, it is necessary that $\dfrac{(1-\alpha)}{B(\alpha)}\bar N 
+ \dfrac{\delta^\alpha}{B(\alpha)\Gamma(\alpha)}\bar N \leq \rho$. 
This implies that
\begin{gather}
\delta<\bigg(\dfrac{\rho B(\alpha)\Gamma(\alpha)}{\bar N}
+\alpha\Gamma(\alpha)-\Gamma(\alpha)\bigg)^{\alpha^{-1}}
\end{gather}
and $\Theta$ maps $W$ into itself.
	
Using the definition of the operator defined in \eqref{eq3.2.1}, we deduce the following:
\begin{equation}
\begin{aligned}
&\|\Theta\Omega_1-\Theta\Omega_2\|_W=\underset{t\in [0, \delta]}{sup}\bigg\|\dfrac{1-\alpha}{B(\alpha)}\{F(t,\Omega_1(t))-F(t,\Omega_2(t))\}\\
&\quad +\dfrac{\alpha}{B(\alpha)\Gamma(\alpha)}\int_{0}^t \{F(p,\Omega_1(p))
-F(p,\Omega_2(p)) \}(t-p)^{\alpha -1}dp \bigg\|_1\\
&\leq \dfrac{1-\alpha}{B(\alpha)}\times\underset{t
\in [0, \delta]}{sup}\big\|F(t,\Omega_1(t))-F(t,\Omega_2(t)) \big\|_1\\
&\quad +\dfrac{\alpha}{B(\alpha)\Gamma(\alpha)}\times\underset{t\in [0, \delta]}{sup}
\int_{0}^t \big\|F(p,\Omega_1(p))-F(p,\Omega_2(p)) \big\|_1(t-p)^{\alpha -1}dp.
\end{aligned}
\end{equation}
In Lemma~\ref{lem3.2.1}, we have shown that $F$ is Lipschitzian 
with respect to the second argument. Then,
\begin{equation}
\begin{aligned}
&\|\Theta\Omega_1-\Theta\Omega_2\|_W
\leq \dfrac{(1-\alpha)L}{B(\alpha)}\times\underset{t\in [0, \delta]}{sup}\big\|\Omega_1(t)-\Omega_2(t)\big\|_1\\
&\quad +\dfrac{\alpha L}{B(\alpha)\Gamma(\alpha)}\int_{0}^\delta \big\|\Omega_1(p)-\Omega_2(p) \big\|_1(t-p)^{\alpha -1}dp\\
&\leq \dfrac{(1-\alpha)L}{B(\alpha)}\times\underset{t\in [0, \delta]}{sup}\big\|\Omega_1(t)-\Omega_2(t)\big\|_1\\
&\quad +\dfrac{\alpha L}{B(\alpha)\Gamma(\alpha)}\int_{0}^\delta \underset{p\in [0, \delta]}{sup}\big\|\Omega_1(p)
-\Omega_2(p) \big\|_1(t-p)^{\alpha -1}dp\\
&\leq \dfrac{(1-\alpha)L}{B(\alpha)}\big\|\Omega_1-\Omega_2\big\|_W 
+\dfrac{\alpha L}{B(\alpha)\Gamma(\alpha)}\int_{0}^\delta \big\|\Omega_1
-\Omega_2\big\|_W(t-p)^{\alpha-1}dp\\
&\leq \bigg(\dfrac{(1-\alpha)L}{B(\alpha)}+\dfrac{ \delta^\alpha L}{B(\alpha)\Gamma(\alpha)}
\bigg) \big\|\Omega_1-\Omega_2\big\|_W.
\end{aligned}
\end{equation}
Thus, the defined operator $\Theta$ is a contraction 
with a unique fixed point $\Omega\in W$  
if $L\bigg(\dfrac{(1-\alpha)}{B(\alpha)}+\dfrac{ \delta^\alpha }{B(\alpha)\Gamma(\alpha)}\bigg) < 1$, 
which implies that
\begin{gather}
\label{condition}
\delta < \bigg(\dfrac{B(\alpha)\Gamma(\alpha)}{L}
+\alpha\Gamma(\alpha)-\Gamma(\alpha)\bigg)^{\alpha^{-1}}.
\end{gather}
This shows that the system under investigation has a unique solution.
\end{proof}		


\section{Stability analysis of Picard's iteration method}
\label{sec4}

In the following, we study the stability of our iterative 
scheme proposed in \eqref{eq3.1.5} to show its convergence.

\begin{definition}[See\cite{ref8}] 
\label{def4.1}
Let $(X,\|\cdot\|)$ be a Banach space and $\Theta$ a self-map of $X$. 
Let $y_{n+1}=g(\Theta,y_n)$ be a particular recursive procedure. 
Suppose that $F_P(\Theta)$ is the fixed-point set of $\Theta$ and has 
at least one element and that $y_n$ converges to a point $p\in F_P(\Theta)$. 
Let $\{x_n\}\subseteq X$ and define $e_n=\|x_{n+1}-g(\Theta,x_n)\|$.
If $\lim\limits_{n \rightarrow +\infty}e_n=0$ implies that 
$\lim\limits_{n \rightarrow +\infty}x_n=p$, then the iteration method 
$y_{n+1}=g(\Theta,y_n)$ is said to be $\Theta$-stable.
\end{definition}

\begin{remark}[See \cite{ref8}] 
\label{rk4}
Without any loss of generality, we must assume that $\{x_n\}$ is bounded. 
Otherwise, if $\{x_n\}$ is not bounded, then it cannot converge. If all  
conditions in Definition~\ref{def4.1} are satisfied for $y_{n+1}=\Theta(y_n)$, 
which is known as Picard's iteration, as a consequence the iteration will be $\Theta$-stable.
\end{remark}

We shall use the following result in the proof of our Theorem~\ref{th4.2.2}. 

\begin{theorem}[See \cite{ref8}] 
\label{th4.2.1}
Let $(X,\|\cdot\|)$ be a Banach space and $\Theta$ 
a self-map of $X$ satisfying
\begin{gather}
\label{condition2}
\|\Theta(x)-\Theta(y)\|\leq K\|\Theta(x)-x\|+k\|x-y\|
\end{gather}
for all $x$, $y$ in $X$ where $0\leq K$, $0\leq k<1$. 
Suppose that $\Theta$ has a fixed point. Then, 
$\Theta$ is Picard's $\Theta$-stable.	
\end{theorem}

\begin{theorem} 
\label{th4.2.2}
The system \eqref{eq3.1.5} is $\Theta$-stable in the Banach space $W$ 
if it satisfies the condition \eqref{condition}.
\end{theorem}

\begin{proof}
Let $n,m\in \mathbb{N}$. We put $\Omega_n=\big(S_{n}, I_{n}, R_{n}\big)$ 
and $\Theta(\Omega_n)=\big(S_{n+1}, I_{n+1}, R_{n+1}\big)$. In order to show 
that $P$ admits a fixed point, we introduce the norm on both sides and we have
\begin{gather}
\begin{aligned}
&\big\|\Theta(\Omega_{n}(t))-\Theta(\Omega_{m}(t))\big\|_1
\leq\dfrac{1-\alpha}{B(\alpha)}\big\|F(t,\Omega_{n}(t))-F(t,\Omega_{m}(t))\big\|_1\\
&\quad+\dfrac{\alpha}{B(\alpha)\Gamma(\alpha)}\int_{0}^t \big\|F(t,\Omega_{n}(t))
-F(t,\Omega_{m}(t))\big\|_1 (t-p)^{\alpha -1}dp.
\end{aligned}
\end{gather}
Then, the same arguments as before lead to
\begin{gather}
\begin{aligned}
\big\|\Theta(\Omega_{n})-\Theta(\Omega_{m})\big\|_W
&\leq L\bigg(\dfrac{1-\alpha}{B(\alpha)}+\dfrac{\delta^\alpha}{B(\alpha)\Gamma(\alpha)}
\bigg)\big\|\Omega_{n}-\Omega_{m}\big\|_W.
\end{aligned}
\end{gather}
If condition \eqref{condition} is satisfied, 
then \eqref{condition2} of Theorem~\ref{th4.2.1} holds. 
This completes the proof.
\end{proof}	


\section{Numerical Simulations}
\label{sec5}

We present numerical simulations of the special solution of our model using 
the Adams--Bashforth--Moulton method presented in \cite{ref27} 
for different arbitrary values of the fractional order $\alpha$.
The convergence and stability of this method apply at very long intervals. 
It is a reliable and a fast alternative for the approximate evaluation 
of Mittag--Leffler functions\cite{ref}. Moreover, it is very useful 
for the numerical evaluation of other special functions arising 
in fractional calculus \cite{ref}. In our numerical computations, 
we consider the following values of the parameters: $\Lambda = 0.7$, 
$\mu = 0.1$, $\lambda=0.02$, $\beta = 0.1$, $r = 0.02$, $d=0.01$, 
$\xi=0.07$, $\gamma=0.1$, $k_1 = 0.1$, $k_2 = 0.02$, and $k_3 = 0.002$ 
with initial conditions $S(0) = 3.0$, $I(0) = 2.0$, $R(0) = 1.0$. 
The numerical results given in Figure~\ref{fig01} show numerical 
simulations of the special solution of our model as a function 
of time for different values of $\alpha$.


\begin{figure}[ht]
\centering
\includegraphics[scale=0.5]{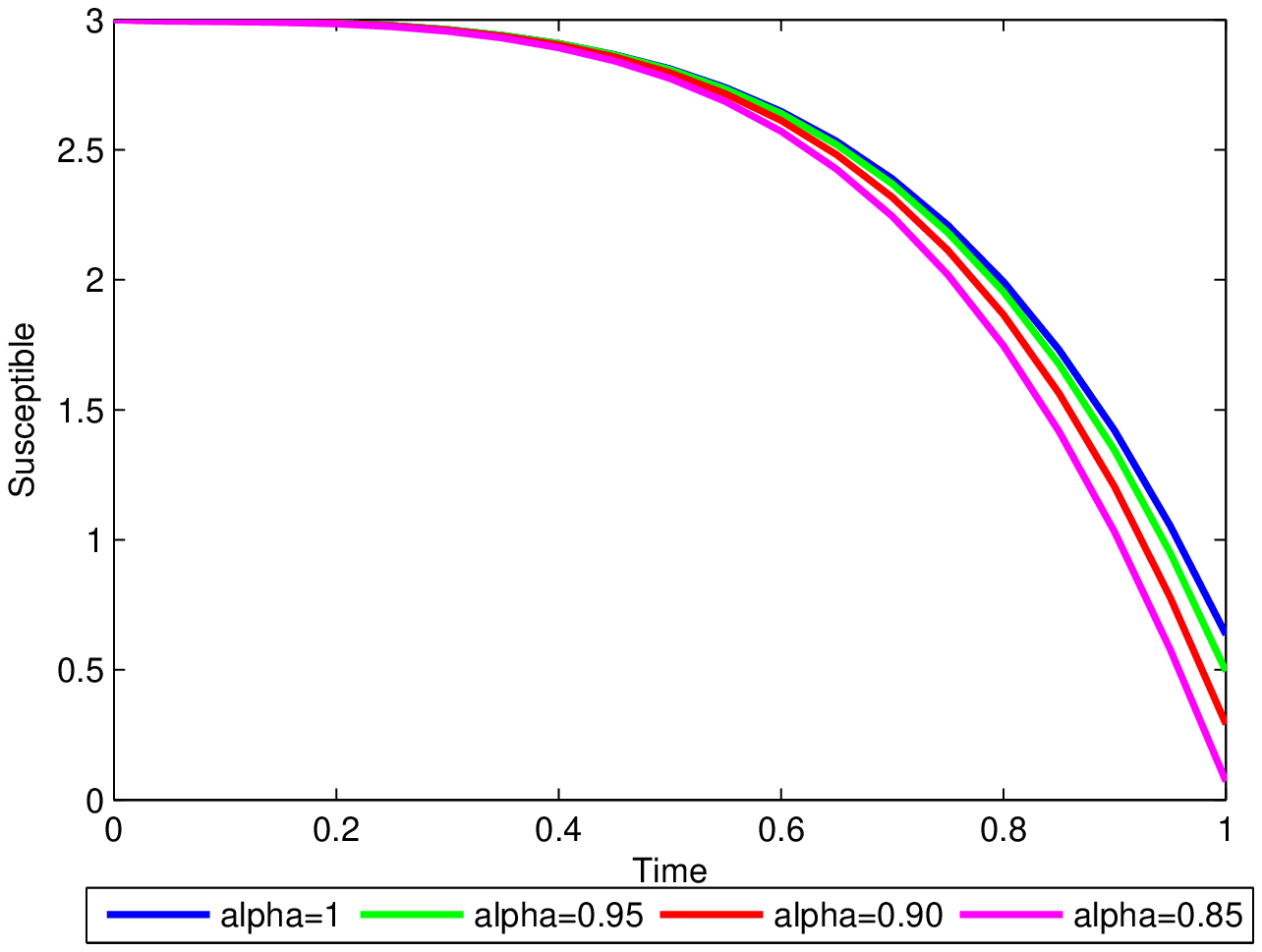}
\centering
\includegraphics[scale=0.5]{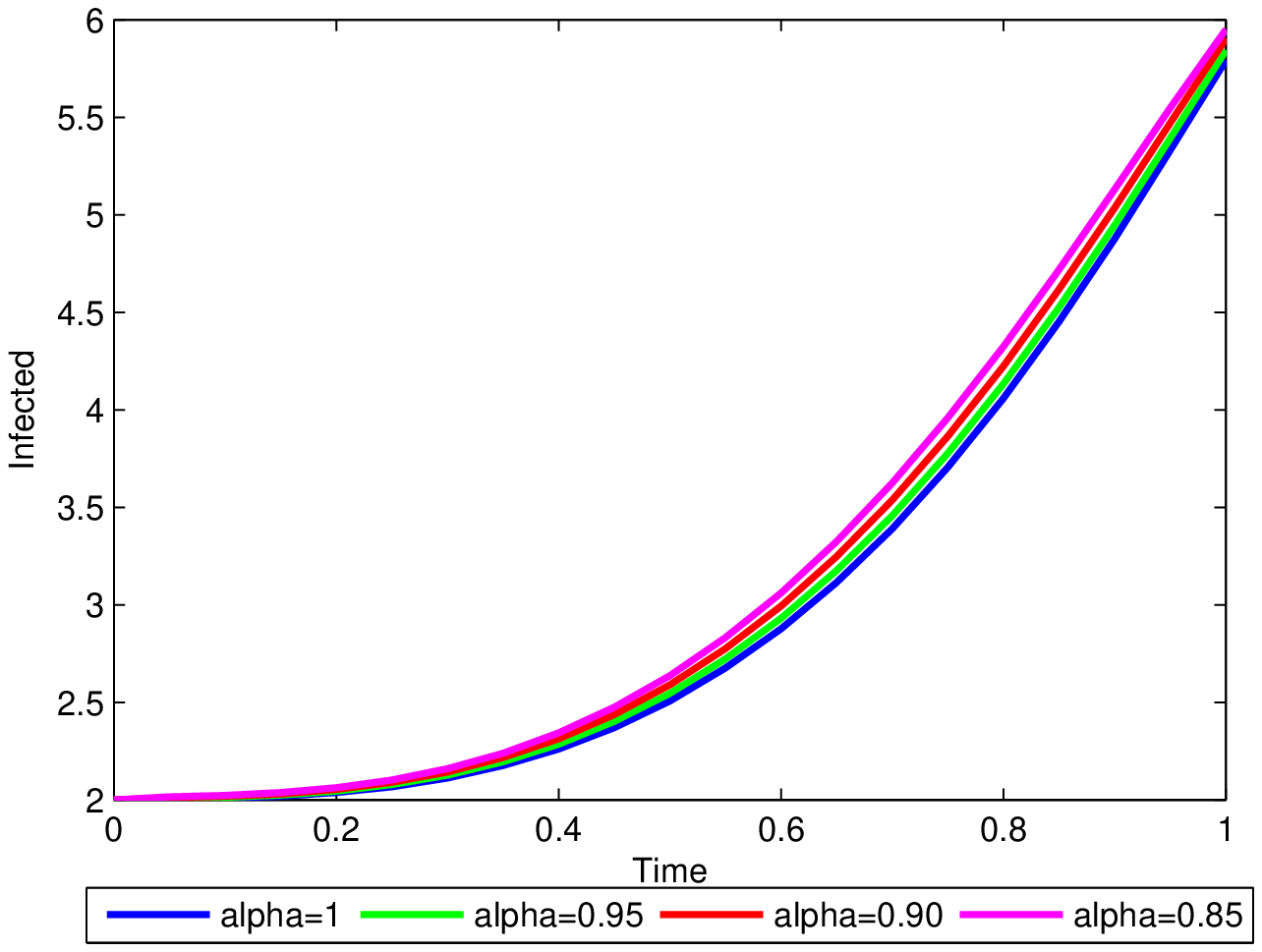}
\centering
\includegraphics[scale=0.5]{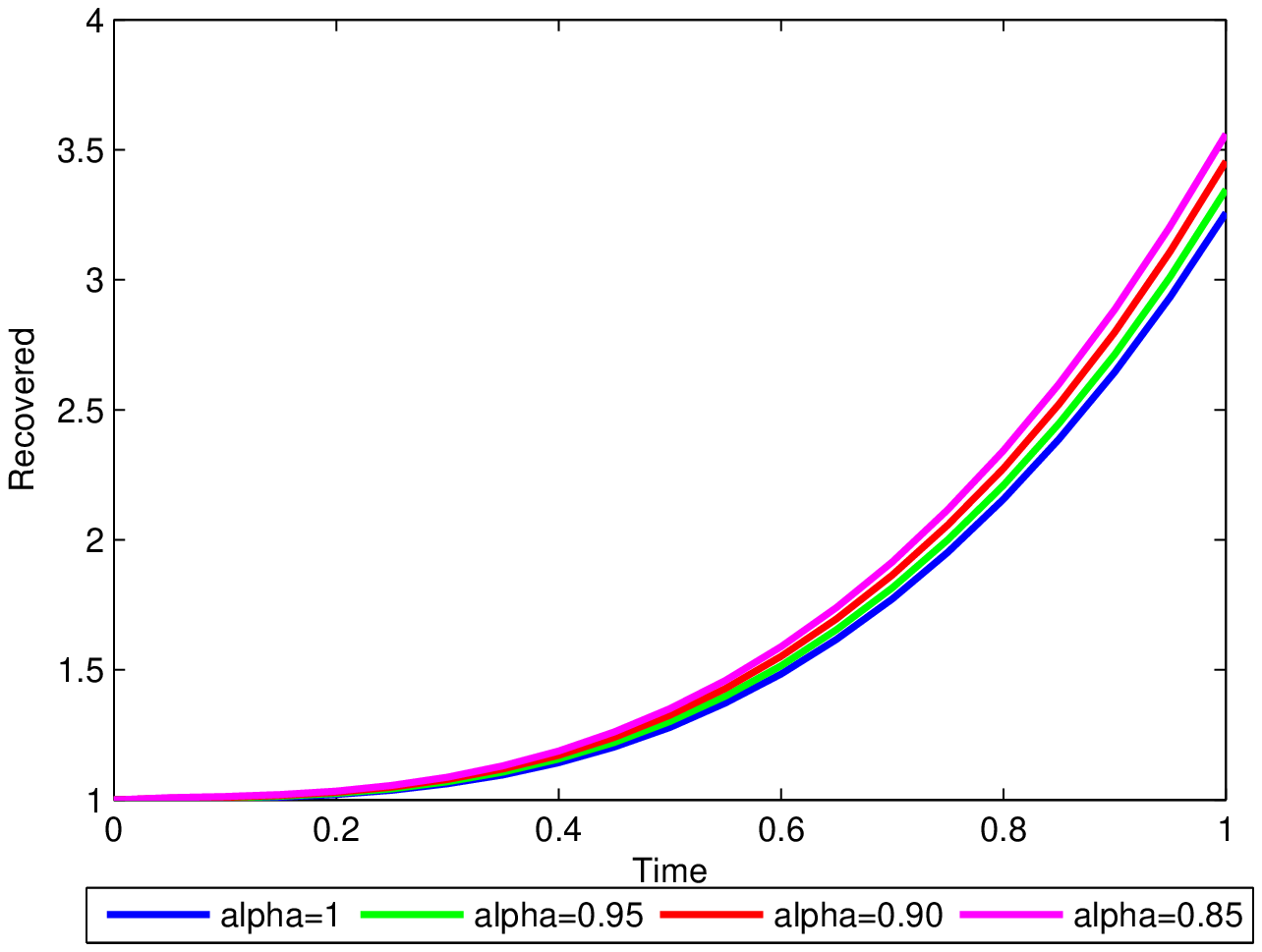}
\caption{Nature of solution with respect to time for different values of $\alpha$.}
\label{fig01}
\end{figure}


Figure~\ref{fig01} shows the dynamics of the interactions between the compartments 
of the susceptible and the infected in a particular environment. We note that when 
the number of susceptible people decreases the number of infected people increases.


\section{Conclusion}

Recently, the fractional derivative of Atangana--Baleanu in the sense of Caputo 
has been used with great success for some applications in different scientific fields. 
This notion of fractional differentiation is based on the Mittag--Leffler function. 
In our work, we extended the SIRS model with a saturated treatment and a nonlinear 
incidence function using the concept of fractional differentiation of Atangana--Baleanu--Caputo. 
We have shown the existence and uniqueness of solution for our model using a fixed point theorem. 
Analysis of the stability of the iterative scheme is validated via the $\Theta$-stable approach. 
Finally, some numerical simulations were presented for different values of $\alpha$.

Since our epidemic model contains a saturated treatment function, 
as future work we plan to study the influence of the parameter 
$ \gamma $ on the stability of our system, which leads to the case 
of multiple equilibrium points and shows bifurcation phenomena.


\section*{Acknowledgements}

Torres was partially funded by FCT,
project UIDB/04106/2020 (CIDMA).


\renewcommand{\bibname}{References}


\end{document}